\documentclass[12pt,a4paper]{amsart} 
\usepackage{amssymb}

\def\Q{\mathbb{Q}}
\def\Z{\mathbb{Z}}
\def\N{\mathbb{N}}

\def\cC{{\mathcal C}}
\def\o{\omega}
\def\G{\Gamma}
\def\FP{{\rm FP}}
\def\<{\langle}
\def\>{\rangle}
\def\-{\overline}

\newtheorem{thmA}{Theorem}

\newtheorem{theorem}{Theorem}[section]
\newtheorem{lemma}[theorem]{Lemma}
\newtheorem{prop}[theorem]{Proposition}
\newtheorem{cor}[theorem]{Corollary}
\newtheorem{corollary}[theorem]{Corollary}

\newenvironment{pf}{\par\medskip\noindent\textit{Proof.}~}{\hfill $\square$\par\medskip}
\newenvironment{pfof}[1]{\par\medskip\noindent\textbf{Proof of #1.}~}{\hfill $\square$\par\medskip}

\date{18 Sept 2007}%MRB,London. Previous dates: 19 Feb, 25 April, 30 April, 14 Sept

\begin{document}

% this gets rid of the AMS series logo and copyright
\catcode`\@=11
\def\serieslogo@{\relax}
\def\@setcopyright{\relax}
\catcode`\@=12

\def\jump{\vskip 0.3cm}
\title[Subgroups of direct products of limit groups]{Subgroups of  direct
products of limit groups}

% author one information
\author[Bridson]{Martin R.~Bridson}
\address{Martin R.~Bridson\\
Mathematical Institute \\
24--29 St Giles'  \\
Oxford OX1 3LB  \\
U.K. }
\email{bridson@maths.ox.ac.uk}

% author two information
\author[Howie]{James Howie }
\address{ James Howie\\
Department of Mathematics and Maxwell Institute for Mathematical Sciences\\
Heriot--Watt University\\
Edinburgh EH14 4AS }
\email{ jim@ma.hw.ac.uk}

% author three information
\author[Miller]{ Charles~F. Miller~III }
\address{Charles~F. Miller~III\\
Department of Mathematics and Statistics\\
University of Melbourne\\
Melbourne 3010, Australia }
\email{ c.miller@ms.unimelb.edu.au }

% author four information
\author[Short]{ Hamish Short }
\address{ Hamish Short \\
L.A.T.P., U.M.R. 6632 \\
Centre de Math\'ematiques et d'Informatique\\
39 Rue Joliot--Curie\\
Universit\'e de Provence, F--13453\\
Marseille cedex 13, France }
\email{ hamish@cmi.univ-mrs.fr }
 
\thanks{This work was supported in part by the CNRS (France)
and by Alliance Scientifique grant \# PN 05.004.
Bridson is supported by a Royal Society Nuffield Research Merit
Award and an EPSRC Senior Fellowship. }



\maketitle

\begin{abstract}
If $\G_1,\dots,\G_n$ are limit groups and
$S\subset\G_1\times\dots\times\G_n$ is of
type $\FP_n(\mathbb Q)$ then $S$ contains a
subgroup of finite index that is itself a direct
product of at most $n$ limit groups. This answers
a question of Sela.
\end{abstract}

\maketitle

\section{Introduction}
The systematic study of the higher finiteness properties of groups
was initiated forty years ago by Wall \cite{wall} and Serre \cite{serre}.
In 1963, Stallings \cite{stall} constructed the first example
of a finitely presented group $\G$ with $H_3(\G;\mathbb Q)$ infinite
dimensional; his example was a subgroup of a direct product of three
free groups. 
This was  the first indication of the
great diversity to be found amongst the finitely presented subgroups
of direct products of free groups, a theme developed in \cite{bieri}.

In contrast, Baumslag and Roseblade \cite{BR}
 proved that in a direct product of two
free groups the only finitely presented subgroups are the 
{\em{obvious ones}}: 
such a subgroup is either free or has a subgroup of finite index
that is a direct product of free groups. 
In \cite{bhms} the present
authors explained this contrast  by proving that
 the exotic behaviour among the finitely presented
subgroups of direct products of free groups is accounted for entirely 
by the failure of higher
homological-finiteness conditions. In particular, we
proved that the only subgroups $S$ of type $\FP_n$ in a 
direct product of $n$ free
groups are the obvious ones: if $S$ intersects each of the direct factors
non-trivially, it virtually splits as the direct product of these intersections.
We also proved that this splitting phenomenon persists when one replaces
free groups by the fundamental groups of compact surfaces \cite{bhms};
in the light of the work of Delzant and Gromov \cite{DG}, 
this has significant implications for  the structure of 
K\"ahler groups.

Examples show that the splitting phenomenon for $\FP_\infty$
subgroups does not extend to products of more general
2-dimensional hyperbolic groups or
higher-dimensional Kleinian groups \cite{mb:haef}. But recent work
at the confluence of logic, group theory and topology has brought to
the fore a class of groups that is more profoundly tied to surface
and free groups than either of the above classes, namely {\em{limit groups}}.

Limit groups arise naturally from several points of view. They are
perhaps most easily described as the finitely generated
groups $L$ that are {\em{fully residually free}} (or 
{\em{$\omega$--residually free}}):
 for any finite subset $T\subset L$ there exists a homomorphism
from $L$ to a free group that is injective on $T$. 
It is in this guise that limit groups were studied extensively by 
Kharlampovich and Myasnikov \cite{KM1,KM2,KM3}.
They are also known as {\em{$\exists$-free groups}} \cite{res}, reflecting the
fact that these are precisely the finitely generated groups that have
the same existential theory as a free group.

More geometrically,  limit groups are the
finitely generated groups that have a Cayley graph in which each
ball of finite radius is isometric to a ball of the same radius in
some Cayley graph of a free group of fixed rank.

The name {\em{limit
group}} was introduced by Sela. His original definition
involved a certain limiting action on an $\mathbb R$-tree, but 
he also emphasized that these are
precisely the groups that arise when one takes limits of 
stable sequences of homomorphisms $\phi_n:G\to F$, where $G$ is an arbitrary
finitely generated group and $F$ is a free group;  {\em{stable}} means that
for each $g\in G$ either $I_g=\{n\in\N : \phi_n(g)=1\}$ or 
 $J_g=\{n\in\N : \phi_n(g)\neq 1\}$ is finite, and the {\em{limit}} of $(\phi_n)$
is the quotient of $G$ by $\{g\mid |I_g|=\infty\}$.

In his account \cite{S9}
of the outstanding problems concerning limit groups, Sela
asked whether the main theorem of \cite{bhms} could be extended to cover
limit groups. The present article represents the culmination of a project
to prove that it can. Building on ideas and results from \cite{bhms,
bh1, bh2, bh3, BM1} we prove:

\begin{thmA}\label{main} If $\G_1,\dots,\G_n$ are limit
groups and $S\subset \G_1\times\cdots\times\G_n$ is a subgroup
of type $\FP_n(\Q)$, then $S$ is virtually a direct product of  $n$
or fewer limit groups.
\end{thmA}

Combining this result with the fact that every finitely generated residually
free group can be embedded into a direct product of finitely many limit groups
(\cite[Corollary 2]{KM2}, \cite[Claim 7.5]{S1}), we obtain:

\begin{corollary}\label{FPinfty} Every residually free group of type
$\FP_\infty(\Q)$ is virtually a direct product of a finite number of limit groups.
\end{corollary}

B.~Baumslag \cite{benjy} proved that a 
finitely generated,
residually free group is fully residually free (i.e.~a limit group)
unless it contains a subgroup isomorphic to $F\times\Z$, where $F$
is a free group of rank 2. Corollary
\ref{FPinfty} together with the methods used to prove Theorem \ref{main}
yield the following generalization of Baumslag's result:

\begin{corollary} Let $\G$ be a residually free group of type $\FP_n(\Q)$
where $n\ge 1$, let $F$ be a free group of rank 2 and let
$F^n$ denote the direct product of $n$ copies of $F$. Either $\G$
contains a subgroup isomorphic to $F^n\times\Z$ or else $\G$ is
virtually a direct product of $n$ or fewer limit groups.
\end{corollary}

We also prove that if a subgroup of a direct product 
of $n$ limit groups 
fails to be of type $\FP_n(\Q)$, then
one can detect this failure in the homology of a subgroup of finite
index. 

\begin{thmA}\label{split}
Let $\G_1,\dots,\G_n$ be limit groups 
and  let $S\subset \G_1\times\cdots\times\G_n$ be a 
finitely generated subgroup
with 
$L_i=\G_i\cap S$ non-abelian for $i=1,\dots,n$.

If $L_i$ is finitely generated for $1\le i\le r$ and not finitely
generated for $i>r$, then there is a subgroup of finite index $S_0\subset S$
such that $S_0=S_1\times S_2$, where $S_1$ is the direct
product of the limit groups $S_0\cap\G_i,\ i\le r$ and (if $r<n$) 
$S_2=S_0\cap(\G_{r+1}\times\cdots\times\G_n)$ has
$H_k(S_2;\mathbb Q)$  infinite dimensional for some $k\le n-r$.
\end{thmA}

In Section \ref{s:last} we shall prove a more technical version
of Theorem \ref{split} and account for abelian intersections.
  
 Theorems \ref{main} and \ref{split} are 
the exact analogues  of  Theorems A and B of \cite{bhms}. 
In Section  \ref{reductions} 
we introduce a sequence of reductions that will allow
us to deduce both theorems from the following result 
(which, conversely, is 
an easy consequence of Theorem \ref{split}). We remind the reader
that a subgroup of a direct
product is called a {\em{subdirect product}} if its projection to each 
factor is surjective. 

\begin{thmA}\label{main3} Let $\G_1,\dots,\G_n$ be non--abelian limit groups 
and  let $S\subset \G_1\times\cdots\times\G_n$ be a 
finitely generated subdirect product which intersects each factor non-trivially. 
Then either :
\begin{enumerate}
\item
$S$ is of finite index; {\em{or}}

\item $S$ is of infinite index and has a finite-index subgroup 
$S_0 < S$ such that $H_j(S_0;\Q)$ has infinite $\Q$-dimension
for some $j\leq n$.
\end{enumerate}
\end{thmA}

For simplicity of exposition, the homology of a group $G$ in this
paper will almost always be with coefficients in a $\Q\,G$-module -- typically the
trivial module $\Q$.  But with minor modifications, 
our arguments also apply
with other coefficient modules, giving corresponding results under
the finiteness conditions $\FP_n(R)$ for other suitable rings $R$. 

A notable aspect of the proof of the above theorems is that following
a raft of reductions based on geometric methods, the proof takes
an unexpected twist in the direction of nilpotent groups.
The turn of events that leads us in
this direction is explained in Section \ref{s:elements}
-- it begins with a
simple observation about higher commutators from \cite{BM1}
and proceeds via a spectral sequence argument.

Several of our results shed light on the nature of arbitrary finitely
presented subgroups of direct products of limit groups, most notably
Theorem 4.2. These results suggest that there is
a real prospect of understanding all such subgroups, i.e.~all
finitely presented residually free groups. We  take up this
challenge in  \cite{bhms3}, where we describe precisely which subdirect
products of limit groups are finitely presented and present a solution
to the conjugacy and membership problems for these subgroups
(cf.~\cite{BM1}, \cite{BW}). But
the isomorphism problem for finitely presented residually free groups remains
open, and beyond that lie many further challenges.
For example, with applications of the surface-group case to K\"ahler 
geometry in mind \cite{DG},
one would like to know if all finitely presented subdirect
products of limit groups satisfy a polynomial isoperimetric inequality. 

\iffalse
Such an
understanding certainly entails a calculation of the 
Bieri-Neumann-Strebel invariants
of direct products of limit groups -- see for example \cite{meinert,MMvW,gehrke}.
The BNS invariants, however, apply only to subgroups that contain
the commutator subgroup, so further methods will need to be developed
to cope with the general case, and many challenges remain. 
\fi

\smallskip

We are grateful to the many colleagues with whom we have had useful
discussions at various times about aspects of this work, particularly
Emina Alibegovi\'c, Mladen Bestvina, Karl Gruenberg, Peter Kropholler, Zlil Sela
and Henry Wilton. We are also grateful to the anonymous referee
for his/her helpful comments.

\section{Limit groups and their decomposition}

Since this is the fourth in a series of papers on limit groups
(following \cite{bh1,bh2,bh3}), we shall only recall the minimal
necessary amount of information about them. The reader unfamiliar
with this fascinating class of groups should consult 
the introductions in \cite{AB,BF}, the original papers of
Sela \cite{S1,S2,S6},  or those of
 Kharlampovich and Myasnikov \cite{KM1,KM2,KM3} where
the subject is approached from a perspective more in keeping with
traditional combinatorial group theory; a further
 perspective is developed in \cite{CG}.
 
\subsection{Limit groups}

Our results rely  on  the fact that limit groups
are the finitely generated subgroups of 
$\omega$-residually free tower ($\o$-rft)
groups 
\cite[Definition 6.1]{S1} (also known as NTQ-groups
\cite{KM2}).  A useful summary of Sela's proof
of this result was given by Alibegovi\'c and Bestvina in
the appendix to \cite{AB} (cf.~\cite[(1.11), (1.12)]{S2}).
The equivalent result of Kharlampovich and Myasnikov
 \cite[Theorem 4]{KM2} is presented in a more algebraic manner. 

An $\o$-rft group is the fundamental group of a
tower space  assembled from graphs, tori
and surfaces in a hierarchical manner. The number of
stages in the process of assembly is the {\em{height}}
of the tower. Each stage in the construction involves
the attachment of an orientable surface along its boundary, or
the attachment of an $n$-torus $T$  along an
embedded circle representing a primitive element
of $\pi_1T$. (There are additional constraints in each
case.)

The {\em{height}} 
of a limit group $\G$ is the minimal height of an 
$\o$-rft group
that has a subgroup isomorphic to $\G$. Limit
groups of height $0$ are free products of 
finitely many free abelian groups (each of finite rank)
and surface
groups of Euler characteristic at most $-2$.

The splitting described in the following proposition is obtained as follows:
embed $\G$ in an $\omega$-rft group $G$ of minimal height, take  the
graph of groups decomposition of $G$  that   the Seifert-van Kampen
Theorem associates to the addition of the final block in the tower, 
then apply Bass-Serre theory to get an induced graph of 
groups decomposition  of $\G$.

Recall that a graph-of-groups decomposition is termed
{\em{$k$-acylindrical}} if in the action on the associated
Bass-Serre tree, the stabilizer of each geodesic edge-path of length
greater than $k$ is trivial; if the value of $k$ is unimportant,
one says simply that the decomposition is {\em{acylindrical}}.

\begin{prop}\label{graph} If $\G$ is a freely-indecomposable
limit group of height $h\ge 1$, then it is the fundamental group of a
finite graph of groups that has infinite cyclic edge groups and has a
vertex group that is a non-abelian limit group of height $\le h-1$.
This decomposition 
may be chosen to be 2-acylindrical.
\end{prop}
 
Note also that any non-abelian limit group of height $0$  splits
as $A\ast_C B$ with $C$ infinite-cyclic or trivial, and this splitting
is 1-acylindrical for surface groups, and 0-acylindrical for free products.
 
\subsection{The class of groups $\cC$}

We define a class of finitely presented groups
$\cC$ in a hierarchical manner; it is the union of the
classes $\cC_n$ defined as follows.

At {\em{level $0$}} we have the class $\cC_0$ consisting
of free products $A\ast B$ of non-trivial,
finitely presented groups, where at least one of $A$ and $B$
has cardinality at least $3$ -- in other words, all finitely
presented non-trivial free products with the exception of $\Z_2*\Z_2$. 
 
A group lies in $\cC_n$ if and only if it is the fundamental group
of a finite, acylindrical graph of finitely presented groups, where all of the
edge groups are cyclic, and at least one of the vertex groups
lies in $\cC_{n-1}$.

\bigskip

The following is an immediate consequence of Proposition \ref{graph}.

\begin{cor}\label{LinC}
All non-abelian limit groups lie in $\cC$.
\end{cor}

\subsection{Other salient properties}\label{ss:props}

In the proof of Theorems \ref{main} and \ref{split}, the
only properties of limit groups $\G$ that will be needed are the
following.
\begin{enumerate}
\item Limit groups are finitely presented, 
and their finitely generated subgroups
are limit groups.
\item 
If $\G$ is non-abelian, it lies in $\cC$ (Corollary \ref{LinC}).
\item Cyclic subgroups are
closed in the profinite topology on $\G$.
 (This is true for
all finitely generated subgroups \cite{wilt}.)
\item If a subgroup $S$ of $\G$ has finite dimensional $H_1(S;\mathbb Q)$,
then  $S$ is finitely generated (and hence is a 
limit group) \cite[Theorem 2]{bh1}.
\item Limit groups are of type $\FP_\infty$ (in fact $\mathcal F_\infty$).
This follows directly from the fact that the class of limit groups
coincides with the class of constructible limit groups, \cite[Definition 1.14]{BF}.
\end{enumerate}

\subsection{Notation}
 Throughout this paper, we consistently use the notational convention that
$S$ is a subgroup of the
direct product of the limit groups $\G_i$ ($1\le i\le n$),  that
$L_i$ denotes the intersection $S\cap\G_i$, and that
$p_i : \G_1\times\dots\times \G_n\to\G_i$ is the coordinate
projection.

\subsection{Subgroups of finite index} \label{ss:fi}
Throughout the proof Theorems \ref{main}, \ref{split}
and \ref{main3}
we shall repeatedly pass to subgroups of finite
index $H_i\subset\G_i$. When we do so, we shall assume
that our original subgroup
$S$
is replaced by $p_i^{-1}(H)\cap S$ and 
that each $\G_j$ ($j\neq i$) is replaced
by $p_jp_i^{-1}(H_i)$. This does not affect the intersections
$L_j= S \cap \G_j$ ($j\neq i$).

Recall \cite[VIII.5.1]{ksbrown} that the property $\FP_n$ 
is inherited by finite-index subgroups and persists in finite
extensions. In the proof of Theorem \ref{main3}
we detect the failure of property
$\FP_n$ by considering the homology of subgroups of finite
index: if $H_k(S_1;\Q)$ is infinite dimensional for some $S_1<S$
of finite index, then neither $S$ nor $S_1$ is of type $\FP_k$.

Some care is required here because one cannot conclude
in the previous sentence that $S$ has an infinite-dimensional
homology group: the finite-dimensionality
of homology groups is a property that persists in finite extensions 
but is not, in general, inherited by finite-index subgroups. 
In  the context of the proof
of Theorem \ref{main3}, care has been taken to ensure
that each passage to a finite-index subgroup
respects this logic.

\section{Reductions of the main theorem}\label{reductions}

The following proposition reduces Theorem \ref{main} 
to Theorem \ref{main3}.

\begin{prop}\label{assume}
Theorem \ref{main} is true if and only if it holds under the following
additional assumptions.
\begin{enumerate}
\item $n\ge 2$.
\item Each projection $p_i:S\to\Gamma_i$ is surjective.
\item Each intersection $L_i=S\cap\Gamma_i$ is non-trivial.
\item Each $\Gamma_i$ is a non-abelian limit group.
\item Each $\Gamma_i$ splits as an HNN-extension over a cyclic subgroup $C_i$
with stable letter $t_i\in L_i$.
\end{enumerate}
\end{prop}

\begin{pf}
\noindent (1) The case $n=0$ of Theorem \ref{main} is trivial.  

In the case $n=1$, $S<\Gamma_1$ has type $\FP_1(\Q)$, so is finitely generated.
But a finitely generated subgroup of a limit group is again a limit group,
 and there is nothing more to prove.
(The case $n=2$ was proved in \cite{bh3} but an independent proof is
given below.)
 
\smallskip\noindent (2)  Since $S$ has type $\FP_n(\Q)$ it is finitely generated,
hence so is $p_i(S)$  and we can  replace  each $\G_i$ by $p_i(S)$.
 
\smallskip\noindent (3) If, say, $L_n$ is trivial, then the projection map
$q_n:S\to\Gamma_1\times\cdots\times\Gamma_{n-1}$ is injective, and $S$
is isomorphic to a subgroup $q_n(S)$ of $\Gamma_1\times\cdots
\times\Gamma_{n-1}$.
After iterating  this argument,  we may assume that each
$L_i$ is non-trivial.

\smallskip\noindent (4) Suppose that one or more of the $\Gamma_i$ is abelian. 
A group is an abelian limit group if and only if it is free abelian of finite
rank.  Hence a direct product of finitely many abelian limit groups is again
an abelian limit group.  This reduces us to the case where precisely
one of the $\Gamma_i$ -- say $\Gamma_n$ -- is abelian.

Now, replacing $\Gamma_n$ by a finite index subgroup if necessary, we may
assume that $L_n\subset\Gamma_n$ is a direct factor of $\Gamma_n$:
say $\Gamma_n=L_n\oplus M$.  Since $M\cap S$ is trivial, the projection
$\Gamma_1\times\cdots\times\Gamma_n\to\Gamma_1\times\cdots
\times\Gamma_{n-1}\times L_n$ 
with kernel $M$ maps $S$ isomorphically onto a subgroup $T$
of $\Gamma_1\times\cdots\times\Gamma_{n-1}\times L_n$.  Since 
$L_n\subset T$, it follows that $S\cong T=U\times L_n$ for some
subgroup $U$ of $\Gamma_1\times\cdots\times\Gamma_{n-1}$.  
But then $U$
has type $\FP_n(\Q)$, since $S$ does, and if Theorem \ref{main} holds
in the case where all the $\Gamma_i$ are non-abelian, then $U$ is virtually
a direct product of $n-1$ or fewer limit groups.  But then $S\cong U\times L_n$
is virtually a direct product of $n$ or fewer limit groups, so Theorem \ref{main}
holds in full generality.

\smallskip\noindent (5)  
The subgroup $L_i$ of $\Gamma_i$ is normal by (2) and non-trivial by (3).
Hence it contains an element
$t_i$ that acts hyperbolically on the tree of the splitting described in 
Proposition \ref{graph}
(see \cite[Section 2]{bh2}).  
Then by
\cite[Theorem 3.1]{bh2}, $t_i$ is the stable letter in some HNN decomposition
(with cyclic edge-stabilizer)  of a finite-index subgroup $\Delta_i\subset\Gamma_i$.

Replacing each $\Gamma_i$ by the corresponding subgroup $\Delta_i$, and
$S$ by $S\cap (\Delta_1\times\cdots\times\Delta_n)$, gives us the desired
conclusion. 

(The above argument extends to all groups in $\cC$ under the additional hypothesis that
the edge groups in the splittings defining $\cC$  are all closed in the profinite topology.)
 \end{pf}

\section{The elements of the proof of Theorem \ref{main3}}\label{s:elements}

We have seen that Theorem \ref{main} follows  from
Theorem \ref{main3}.
The proof of Theorem \ref{main3} extends from Section   \ref{fgnormal} to 
Section \ref{finalstep}. 
In the present section we give an overview of the contents of these sections
and indicate how they will be assembled to complete the proof.

In Section \ref{fgnormal} we prove the following extension of the basic result that
non-trivial, finitely-generated normal subgroups of non-abelian limit groups have
finite index \cite{bh1}.

\begin{theorem}\label{index}
Let $\Gamma$ be a group in $\cC$, and $1\ne N<G<\Gamma$ with $N$ normal 
in $\Gamma$ and $G$ finitely generated. 
 Then $|\Gamma:G|<\infty$. 
\end{theorem}

Using this result, together with the HNN decompositions of the
$\Gamma_i$ described in Proposition \ref{assume}, 
we deduce (Section \ref{vna}):

\begin{theorem}\label{propvna}
Let $\G_1,\dots,\G_n$ be non-abelian limit groups.
If $S\subset \G_1\times\cdots\times\G_n$ 
is a finitely generated subgroup with $H_2(S_1;\Q)$ finite dimensional
for all finite-index subgroups $S_1<S$,
and if $S$
satisfies  conditions (1) to (5) of Proposition \ref{assume},
then:

\begin{itemize}
\item the image of each projection $S\to\G_i\times\G_j$ 
is of finite index in $\G_i\times\G_j$;
\item the quotient groups
$\Gamma_i/L_i$ are virtually nilpotent
of  class at most $n-2$.
\end{itemize}
\end{theorem}

We highlight the case $n=2$.

\begin{cor}\label{cor2factor} 
If $\Gamma_1$ and $\Gamma_2$ are non-abelian limit groups, 
and $S<\Gamma_1\times\Gamma_2$
is a subdirect product  intersecting each factor non-trivially, 
with $H_2(S_1;\Q)$ finite dimensional
for all finite-index subgroups $S_1<S$,
 then $S$ has finite index in $\Gamma_1\times\Gamma_2$.
\end{cor}

An important special case of Theorem \ref{main3},
considered in Section \ref{kernelZ},
arises where $S$ is the kernel of an epimorphism
$\Gamma_1\times\cdots\times\Gamma_n\to\Z$.

\begin{theorem}\label{theoremkernelZ}
Let $\Gamma_1,\dots,\Gamma_n$ be non-abelian limit groups, and
$N$ the kernel of an epimorphism
$\Gamma_1\times\cdots\times\Gamma_n\to\Z$.  
Then there is a subgroup of finite index $N_0\subset N$
such that 
at least one of the homology groups $H_k(N_0;\Q)$ ($0\le k\le n$)
has infinite $\Q$-dimension. 
\end{theorem}

We complete the proof of Theorem \ref{main3} in Section \ref{finalstep}.
We have seen that each of the $\Gamma_i/L_i$ is virtually nilpotent.
Setting $\Gamma=\Gamma_1\times\cdots\times\Gamma_n$ and 
noting that $S$ contains the product $L=L_1\times\cdots\times L_n$,
we argue by induction on the difference in Hirsch lengths $d=h(\Gamma/L)-h(S/L)$
to prove that $H_k(S;\Q)$ has infinite $\Q$-dimension for some $k\le n$
if $d>0$. 
The initial step of the induction is provided by Theorem \ref{theoremkernelZ},
and the inductive step is established using the LHS spectral sequence.
Section \ref{s:last} contains a proof of Theorem \ref{split}.

\section{Subgroups containing normal subgroups}\label{fgnormal}
In this section we prove Theorem \ref{index}.
We assume that the reader is familiar with Bass-Serre theory
\cite{Serre2},
which we shall use freely. 
All our actions on trees are without inversions.

\begin{lemma}\label{lemmacocompact}
Let $\Delta$ be a group acting $k$-acylindrically, 
cocompactly and minimally on a tree $X$.
Let $H$ be a finitely generated subgroup of $\Delta$.
Suppose that $M<H$ is a non-trivial subgroup which is normal in $\Delta$.
Then the action of $H$ on $X$ is cocompact.
\end{lemma}

\begin{proof}
If $X$ is a point there is nothing to prove, so we may assume that
$X$ has at least one edge.
By hypothesis, $\Delta$ has no global fixed point in its action on $X$.
By \cite[Corollary 2.2]{bh2}, the non-trivial normal subgroup $M<\Delta$ contains
elements which act hyperbolically on $X$, and the union of the axes of all such elements
is the unique minimal $M$-invariant subtree $X_0$ of $X$.
Since $M$ is normal in $\Delta$, the $M$-invariant subtree $X_0$ is also
invariant under the action of $\Delta$.  But $X$ is minimal as a $\Delta$-tree,
so $X_0=X$.

We have shown that $M$ acts minimally on $X$. Since $M<H$, it
follows that $H$ acts minimally on $X$, so the quotient graph of
groups $\mathcal G$ has no proper sub-(graph of groups) such that the inclusion induces an
isomorphism on $\pi_1$. A standard argument in Bass-Serre theory shows that
since $H$ is finitely
generated, the topological graph underlying $\mathcal G$ is compact, as claimed.
\end{proof}

\begin{prop}\label{propindex}
Let $\Gamma\in\cC$, and let $C,G$ be 
subgroups of $\Gamma$ with $C$ cyclic and $G$ finitely
generated.  If $|G\backslash\Gamma/C|<\infty$,
then $|\Gamma:G|<\infty$.
\end{prop}

\begin{pf}
Let $\Gamma$ be a group in $\cC$.
We argue by induction on the level $\ell=\ell(\Gamma)$ in the hierarchy 
$\mathcal C=\cup_n{\mathcal C}_n$ where $\G$ first appears. 
By definition,  $\Gamma$ has a non-trivial,
$k$-acylindrical, cocompact
action on a tree $T$, with cyclic edge stabilizers.
Without loss of generality we can suppose that this action is minimal.
 
If $\ell=0$ there is a single orbit of edges, the edge stabilizers are trivial and the vertex stabilizers are non-trivial.
If $\ell>0$ the edge-stabilizers are non-trivial and the stabilizer of some vertex $w$ is in ${\mathcal C}_{\ell-1}$.

Let $c$ be a generator for $C$.
We treat the initial and inductive stages of the argument simultaneously, 
but distinguish two cases according to the action
of $c$.

\smallskip\noindent{\bf Case 1.} Suppose that $c$ fixes a vertex $v$ of $T$.

Then, by our double-coset hypothesis, the $\Gamma$-orbit of $v$ consists of only
finitely many $G$-orbits $Gv_i$.  Since the action of $\G$ on $T$
is cocompact, there is a constant $m>0$ such that $T$ is the $m$-neighbourhood
of $\Gamma v$, and hence the quotient graph $X=G\backslash T$ is the 
$m$-neighbourhood of the finitely many vertices $Gv_i$. In other
words,  $X$ has finite diameter.  

Note also that $\pi_1X$ has finite rank, because it is a retract
of $G$ which is finitely
generated.

Finally, note that $X=G\backslash T$ has only finitely many valency 1 vertices.  
For otherwise, we can deduce a contradiction
as follows.  Since $G$ is finitely generated, if there are infinitely
many vertices of valency 1, then the induced graph-of-groups
decomposition of $G$ is degenerate, in the sense that there
is a valency 1 vertex $\bar u$ with $G_{\bar u}=G_{\bar e}$,
where $\bar e$ is the unique edge of $G\backslash T$
incident at $\bar u$.

Now $\bar u=Gu$ for some vertex $u$ in $T$,
 and $\bar e=Ge$ for an edge $e$ incident at $u$.
The group $G_{\bar u}$
is the stabilizer of $u$ in $G$, and $G_{\bar e}$ is the stabilizer
of $e$ in $G$.  The fact that $\bar u=Gu$ has valency $1$ in 
$G\backslash T$ means that $G_{\bar u}$ acts transitively on the link
{\rm{Lk}} of $u$ in $T$.  Hence $|{\rm{Lk}}|=|G_{\-u}:G_{\-e}|=1$, so
$u$ is a valency
$1$ vertex of $T$.  But this contradicts the fact that $T$ is minimal
as a $\Gamma$-tree.

We have shown that $X=G\backslash T$ has finite diameter, finite rank,
and only finitely many vertices of valency 1.  It follows that $X$
is a finite graph. 

In the case where $\G$ has level $\ell=0$, the stabilizer
$\G_e$ of any edge $e$ of $T$ is trivial. The number of edges
in $X=G\backslash T$ that are images of edges $\gamma e\in\Gamma e$
can therefore be counted as $|G\backslash\G/\G_e|=|G\backslash\G|=|\G:G|$.
Hence, in this case, $|\G:G|<\infty$, as required.

In the case where $\ell>0$, there is a vertex $w$ of $T$ whose
stabilizer $\G_w$ in $\G$ is a group in $\cC_{\ell-1}$.
Let $\Gamma_e$ denote the stabilizer of some edge
$e$ incident at $w$.  Then
$|(G\cap\Gamma_w)\backslash\Gamma_w/\Gamma_e|$ is bounded above
by the finite number of edges of
$X=G\backslash T$ incident at $Gw\in G\backslash T$
that are images of edges $\gamma e\in \Gamma e$. 
By inductive hypothesis, $G\cap\Gamma_w$ has finite index in
$\Gamma_w$.  Similarly, for each $\gamma\in\Gamma$,
$G\cap\gamma\Gamma_w\gamma^{-1}$ has finite index in
$\gamma\Gamma_w\gamma^{-1}$.  
Consider the action of $\G_w$ by right multiplication on $G\backslash\G$: the orbits
are
the double cosets $G\backslash\G/ \G_w$ and hence are finite in number
because they index a subset of the vertices of $X=G\backslash T$; moreover the
stabilizer of $G\gamma$ is $\gamma^{-1} G \gamma \cap \G_w$, which we have just
seen has finite index in $\G_w$.  Thus $G\backslash\G$ is finite.

\smallskip\noindent{\bf Case 2.}
Suppose that $c$ acts hyperbolically on $T$, with
axis $A$ say. 

Then the double coset hypothesis implies that the axes
$\gamma(A)$, for $\gamma\in\Gamma$, belong to only finitely many 
$G$-orbits.   On the other hand, the convex hull of 
$\bigcup_{\gamma\in\Gamma}\gamma(A)$ is a $\Gamma$-invariant subtree of $T$, 
and hence by minimality is the whole of $T$.

Let $T_0$ be the minimal $G$-invariant subtree of $T$.  If $T_0=T$
then $X=G\backslash T$ is finite since $G$ is finitely generated, and so
$|G\backslash\Gamma/\Gamma_e|<\infty$ for any edge-stabilizer 
$\Gamma_e$ in $\Gamma$.  
If $\ell=0$, then $\G_e$ is trivial, so $|\G:G|<\infty$.  
Otherwise, choose  
 $e$ incident at a vertex $w$ whose stabilizer $\G$
 is in $\cC_{\ell-1}$ and apply the inductive hypothesis as above
 to deduce that $|\Gamma:G|<\infty$.

It remains to consider the case $T_0\ne T$.

Now, for any subgraph $Y$ of $T$, and any $g\in G$, 
we have $$d(g(Y),T_0)=d(g(Y),g(T_0))=d(Y,T_0).$$
Since the $\Gamma$-orbit of $A$ contains only finitely many
$G$-orbits, there is a global upper bound $K$, say, on
$d(\gamma(A),T_0)$ as $\gamma$ varies over $\G$.

Since $T\neq T_0$ and $T$ is spanned by the $\Gamma$-orbit of $A$, 
there is a 
translate $\gamma(A)$ of $A$ that is not contained in $T_0$.  
Recall that the action is $k$--acylindrical.
Choose a vertex $u$ on $\gamma(A)$ with $d(u,T_0)>K+k+2$
and let $\G_u$ denote its stabiliser in $\G$. Let $p$ be the
vertex a distance $K$ from $T_0$ on the unique shortest 
path from $T_0$ to $u$. Since $d(\gamma(A),T_0)\le K$,
the geodesic $[p,u]$ is contained in $\gamma(A)$. Similarly,
$[p,u]$ is contained
in any translate of $A$ that passes through $u$.
In particular, if $\delta\in\Gamma_u$ 
then $[p,u]\subset\delta\gamma(A)$, and since $\delta$
fixes $u$ we have $\delta(p)=p$ or $\delta(p')=p$,
where $p'$ is the unique point of $\gamma(A)$ other than
$p$ with $d(u,p)=d(u,p')$. 

If $\delta$ fixes the edge of $[p,u]$
incident at $u$, then $\delta(p)=p$ hence $\delta$ fixes
$[p,u]$ pointwise, which  contradicts the $k$-acylindricality of
the action unless $\delta=1$. Thus the stabiliser of this edge
is trivial, which is a contradiction  unless $\ell=0$.

If $\ell=0$ then, replacing $u$ by an adjacent vertex if
necessary, we may assume that $|\Gamma_u|>2$.
Choose distinct non-trivial elements $\delta_1,\delta_2\in\G_u$.
It cannot be that all three of $\delta_1,\delta_2,\delta_1\delta_2^{-1}$
send $p'$ to $p$. Thus one of them fixes $p$, hence
$[p,u]$, which again contradicts the $k$-acylindricality of
the action.

\end{pf}

We are now able to complete the proof of Theorem \ref{index}.

\begin{pfof}{Theorem \ref{index}}
Suppose that $\Gamma\in\cC$, $G<\G$ is finitely generated, and $N$ is a non-trivial
normal subgroup of $\G$ that is contained in $G$. 
 Then by definition of $\cC$, $\G$ acts non-trivially,
 cocompactly and $k$-acylindrically on
a tree $T$ with cyclic edge stabilizers.  There is no loss of generality
in assuming the action is minimal, so we may apply 
Lemma \ref{lemmacocompact} to see that the 
action of $G$ is cocompact.  
The stabilizer $\G_e$ in $\G$
of an edge $e$ is cyclic, and the finite number of edges in
$G\backslash T$ is an upper bound on $|G\backslash\G/\G_e|$.
It follows from Proposition \ref{propindex} that $|\G:G|<\infty$, as claimed.
\end{pfof}

\section{Nilpotent quotients}\label{vna}

In this section we prove Theorem \ref{propvna}, which steers us
away from the study of groups acting on trees and into the realm
of nilpotent groups.

We first prove a general lemma (from \cite{BM1})
about a subdirect product $S$ of $n$ arbitrary (not necessarily limit) groups $\G_1,\dots,\G_n$.  As before, we write $L_i$ for
the normal subgroup $S\cap\G_i$ of $\G_i$.
We also introduce the following notation.  We write $K_i$ for the
kernel of the $i$-th projection map $p_i:S\to\Gamma_i$, 
and $N_{i,j}$ for the image of $K_i$ under the $j$-th projection 
$p_j:S\to\Gamma_j$.  
Thus $N_{i,j}$ is a normal subgroup of $\Gamma_j$.

We shall denote by  $[x_1,x_2,\dots,x_n]$ the  left-normed $n$-fold commutator   
$[[..[x_1,x_2],x_3],\dots],x_n]$.

\begin{lemma}\label{multicomm} 
$[N_{1,j},N_{2,j},\dots,N_{j-1,j},N_{j+1,j},\dots,N_{n,j}]\subset L_j$.
\end{lemma}

\begin{pf} Suppose that $\nu_{i,j}\in N_{i,j}$ for a fixed $j$ and for all $i\ne j$.
Then
 there exist $\sigma_i\in S$ with $p_i(\sigma_i)=1$ and $p_j(\sigma_i)=\nu_{i,j}$.  
Let $\sigma$ denote the $(n-1)$-fold commutator
$[\sigma_1,\dots,\sigma_{j-1},\sigma_{j+1},\dots,\sigma_n]\in S$.  Then
$p_j(\sigma)$ is the $(n-1)$-fold commutator
$$[\nu_{1,j},\dots,\nu_{j-1,j},\nu_{j+1,j},\dots,\nu_{n,j}]\in\Gamma_j.$$

On the other hand, for $i\ne j$, we have $p_i(\sigma)=1$ since $p_i(\sigma_i)=1$.
Hence $\sigma\in L_j$, and $p_j(\sigma)=\sigma\in L_j$.

Since the choice of $\nu_{i,j}\in N_{i,j}$ was arbitrary, we have 
$$[N_{1,j},N_{2,j},\dots,N_{j-1,j},N_{j+1,j},\dots,N_{n,j}]\subset L_j$$
as claimed.
\end{pf}

We now consider a finitely generated
subdirect product $S$ of non-abelian
limit groups $\Gamma_1,\dots,\Gamma_n$
such that  $H_2(S_1;\Q)$ is finite dimensional for every
finite-index subgroup $S_1<S$.
  
Let $L_i,C_i$ and $t_i$ be as in Proposition \ref{assume}.
We consider the image $A_{i,j}:=p_j(p_i^{-1}(C_i))$ 
under the projection
$p_j$ of the preimage under $p_i$ of the cyclic group $C_i$.  
Clearly $N_{i,j}<A_{i,j}<\Gamma_j$.

In the remainder of this section we shall prove that $N_{i,j}\subset \G_j$
is of finite index for all $i$ and $j$. Lemma \ref{multicomm}
 then implies that
$\G_i/L_i$ is virtually nilpotent of class at most $n-2$,
as is claimed in Theorem \ref{propvna}. 

As a first step towards showing that $N_{i,j}\subset \G_j$
is of finite index, we prove the following lemma.

\begin{lemma}\label{comms}
Let $\G_1,\dots,\G_n$  be non-abelian limit groups.
If $S < \G_1\times\cdots\times\G_n$ 
is a finitely generated subgroup
with $H_2(S;\Q)$ finite dimensional, and if $S$
satisfies  conditions (1) to (5) of Proposition \ref{assume},
then for all $i,j$:
\begin{enumerate}
\item $|\Gamma_j:A_{i,j}|<\infty$;
\item $A_{i,j}/N_{i,j}$ is cyclic.
\end{enumerate}
\end{lemma}

\begin{pf}
(1) It suffices to consider the case $i=1$.
The HNN decomposition $\G_1= B_1*_{C_1}$ described in
Proposition \ref{assume} (5) pulls back to an HNN
decomposition of $S$ with stable letter $\widehat t_1=(t_1,1,\dots,1)$, 
base group $\widehat {B_1}={p_1}^{-1}(B_1)$, and
 amalgamating subgroup $\widehat C_1={p_1}^{-1}(C_1)$.
As $C_1$ is cyclic, $\widehat C_1=K_1\rtimes\<\widehat c_1\>$
where $\widehat c_1$ is a choice of a lift of a generator of $C_1$.
Consider the Mayer-Vietoris sequence for the HNN decomposition 
of $S$.
$$\cdots\to H_2(S;\Q)\to H_1(\widehat C_1;\Q) \overset{\phi}\to 
H_1(\widehat{B_1};\Q) \to H_1(S;\Q)\to\cdots$$
The map $\phi$ is the difference between the map 
induced by inclusion and the map  induced by the 
inclusion twisted by the action of $\widehat t_1$ by conjugation.
Notice that $\widehat t_1$ commutes with $K_1$ 
and so acts trivially on $H_*(K_1;\Q)$. Thus $\phi$
factors through the map $H_1(\widehat C_1;\Q)\to H_1(\langle\widehat c_1\rangle;\Q)$,
in particular the image of $\phi$ has dimension at most 1. 
Since $H_2(S;\Q)$ is finite dimensional by hypothesis, 
it follows that $H_1(\widehat C_1;\Q)$ is finite dimensional.
For each $j$, $A_{1,j}=p_j(\widehat C_1)$ is a homomorphic image
of $\widehat{C_1}$, so $H_1(A_{1,j};\Q)$ is finite-dimensional.
Since $A_{1,j}$ is a subgroup of the non-abelian limit group $\G_j$,
it follows that it is finitely generated.  Since it
contains the non-trivial normal subgroup $L_j$,
Theorem \ref{index} now implies that $A_{1,j}$
has finite index in $\Gamma_j$, as claimed.

(2) As $p_j$ is surjective,
  $A_{i,j}/N_{i,j}=p_j(\widehat C_i)/p_j(K_i)$ is a homomorphic
image of $\widehat C_i/K_i$, so it is also cyclic, as claimed.
\end{pf}

The other crucial ingredient in the proof of Theorem \ref{propvna}
is the following proposition.
 
\begin{prop}\label{surprising}
 Let $G$ be an HNN extension of the form
$B\ast_C$ with stable letter $t$, finitely generated base-group
$B$ and infinite-cyclic edge group $C$. Suppose that $G$
has normal subgroups $L$ and $N$ such that $t\in L$,
$C\cap N=\{1\}$ and $G/N$ is infinite-cyclic. Suppose further
that $H_1(N;\Q)$ is infinite dimensional. Let $\Delta\subset G$ be the
unique subgroup of index 2 that contains $B$. Then, there
exists an element $x\in L\cap B\cap N$ such that 
$R\overline x\subset H_1(N\cap \Delta;\Q)$ is a free $R$-module
of rank 1, where $R=\Q[\Delta/(N\cap \Delta)]$ and $\overline x$
is the homology class determined by $x$.
\end{prop}

\begin{proof}  Let $T$ be the Bass-Serre tree of the splitting
$G=B\ast_C$ and consider the graph of groups decomposition
of $N_2:=N\cap \Delta$ with underlying graph $X=N_2\backslash T$; since $N_2C$
has finite index in $G$, this is a finite graph. Each vertex
group in this decomposition is a conjugate of $B\cap N_2$,
and the edge groups are trivial since $C\cap N_2=\{1\}$.

Thus, as an abelian group, $H_1(N_2;\Q)$ is the direct sum of
$H_1(X;\Q)$ and $p$ copies of $H_1(B\cap N_2;\Q)$,
where $p$ is the 
index of $BN_2$ in $G$. The first of these summands is finite-dimensional, and hence $H_1(B\cap N_2;\Q)$ is
infinite-dimensional (since $H_1(N;\Q)$ is infinite-dimensional,
implying that $H_1(N_2;\Q)$ is too). 

Let $\tau$ be a generator of $G/N$. Then $M:=H_1(B\cap N_2;\Q)$ is a
$\Q[\tau^{\pm p!}]$-module, which is finitely generated
because $B$ is finitely generated and $B/(B\cap N_2)$
is finitely presented. 
Since $\Q[\tau^{\pm p!}]$ is a principal ideal domain, the
module $M$ has a free direct summand. We fix $z\in B\cap N_2$
so that $\overline z\in M$ generates this free summand. 
It follows that $R\overline z$ has infinite $\Q$-dimension, and so
is a free submodule of 
the $R$-module $H_1(N_2;\Q)$.
 
 Since $t\notin \Delta$,
 $z_1:=z$ and $z_2:=tzt^{-1}$ belong to distinct vertex groups
 in $X$.  Hence $x:=[z,t]=z_1z_2^{-1}\in L\cap N\cap \Delta$
 is such that $\overline x=\overline z_1-\overline z_2$
 generates a free $\Q[\tau^{\pm p!}]$-submodule of $H_1(N_2;\Q)$, and hence also a free $R$-submodule.
 \end{proof}

The following proposition completes the proof of Theorem \ref{propvna}.

\begin{prop} 
Let $\G_1,\dots,\G_n$  be non-abelian limit groups.
If $S < \G_1\times\cdots\times\G_n$ 
is a finitely generated subgroup
with $H_2(S_1;\Q)$ finite dimensional for each subgroup
$S_1$ of finite index in $S$, and if $S$
satisfies  conditions (1) to (5) of Proposition \ref{assume},
then (in the notation of Lemma
\ref{comms})
$N_{i,j}\subset \G_j$
is of finite index for all $i$ and $j$.
\end{prop}

\begin{pf} It suffices to consider the case $(i,j)=(2,1)$.
Let $T$ be the projection of $S$ to $\G_1\times\G_2$, and define
$M_i=T\cap\G_i$ for $i=1,2$.  Notice that $M_1 = N_{2,1}$, the projection to $\G_1$
of the kernel of the projection $p_2:S\to\G_2$,
and similarly $M_2=N_{1,2}$.  

Since $S$ projects onto each of $\G_1$ and $\G_2$, the same is
true of $T$.  Hence we have isomorphisms
$$\frac{\G_1}{M_1}\cong\frac{T}{M_1\times M_2}\cong\frac{\G_2}{M_2}.$$
We will assume that these groups are infinite, and obtain a contradiction.

By Lemma \ref{comms}, $T/(M_1\times M_2)$ is virtually cyclic,
so we may choose a finite index subgroup $T_0<T$ containing
$M_1\times M_2$ such that $T_0/(M_1\times M_2)$ is infinite
cyclic.  Hence $G_i:=p_i(T_0)$ is a finite-index subgroup
containing $M_i$ for $i=1,2$, such that $G_i/M_i$ is infinite
cyclic.  Choose $\tau\in T_0$ such that $\tau.(M_1\times
M_2)$ generates $T_0/(M_1\times M_2)$, and let 
$\tau_i=p_i(\tau)\in G_i$
for $i=1,2$.

The HNN-decomposition of $\G_i$ from Proposition \ref{assume} (5)
induces
 an HNN decomposition $G_i=B'_i\ast_{C'_i}$
with stable letter $t'_i\in L_i$, where $C'_i=C_i\cap G_i$
and $t'_i$ an appropriate power of the stable letter $t_i$ of
$\G_i$.
Notice that, by Lemma \ref{comms}, $C'_i\cap M_i=\{1\}$.
For each $i=1,2$, Proposition \ref{surprising} 
(with $G=G_i$, $N=M_i$, $L=L_i$, $t=t'_i$, $B=B'_i$,
$C=C'_i$) provides an index $2$ subgroup $\Delta_i$
in $G_i$ and an element $ x_i\in M_i\cap\Delta_i\cap L_i$ such
that $\overline x_i$ generates a free $\Q[\tau_i^{\pm 1}]$-submodule
of $H_1(M_i\cap\Delta_i;\Q)$.

Now define $M'_i:=M_i\cap\Delta_i$.
It follows that  $\overline x_1\otimes \overline x_2$ generates a free $\Q[\tau_1^{\pm 1},\tau_2^{\pm 1}]$-submodule of 
$$H_1(M'_1;\Q)\otimes_\Q H_1(M'_2;\Q)\subset H_2(M'_1\times M'_2;\Q).$$

Let $T_1$ be the finite-index subgroup of $T_0$ defined by
$T_1:=(M'_1\times M'_2)\rtimes\<\tau\>$, and let $S_1<S$ be the 
preimage of $T_1$ under the projection $S\to T$.
Using the LHS spectral sequence for the short exact sequence
$M'_1\times M'_2\to T_1\to\<\tau\>$, we see that 
$$H_0(\langle \tau\rangle;H_2(M'_1\times M'_2;\Q))\subset H_2(T_1;\Q)$$
has an infinite dimensional $\Q$-subspace generated by the images of
$$\{(\tau_1^mx_1\tau_1^{-m})\otimes (\tau_2^nx_2\tau_2^{-n});~m,n\in\Z\}.$$
In particular, the image of the map
$H_2(L_1\times L_2;\Q)\to H_2(T_1;\Q)$ induced by inclusion is infinite-dimensional.
But this contradicts the hypothesis that $H_2(S_1;\Q)$ is finite dimensional,
 since the inclusion $(L_1\times L_2)\to T_1$ factors
through $S_1$. This is the desired contradiction which completes the proof.
\end{pf}

\section{Normal subgroups with cyclic quotient}\label{kernelZ}

\begin{prop} If $\G_1,\dots, \G_n$ are groups of
type $\FP_n(\Z)$ and $\phi:\G_1\times\cdots\times
\G_n\to\Z$ has non-trivial restriction to each
factor, then $H_{j}(\ker\phi;\Z)$
is finitely generated for $j\le n-1$. 
\end{prop}

\begin{proof} We first prove the result in the special
case where the restriction of
$\phi$ to each factor is epic.
Thus we may write $\G_i = L_i\rtimes\langle t_i\rangle$ where
$S=\ker\phi$, $L_i=S\cap\G_i$ is the kernel of $\phi|_{\G_i}$ and $\phi(t_i)$ is a fixed
generator of $\Z$.

If $n\ge 2$
and we fix a finite set $A_i\subset L_i$ such that $\G_i=\langle A_i,t_i\rangle$,
then $S$ is generated by $A_1\cup\dots\cup A_n\cup
\{ t_1t_2^{-1},\dots,t_1t_n^{-1}\}$. 

 We proceed by induction on $n$
(the initial case $n=1$ being trivial), considering
the LHS spectral sequence in homology for
the projection of $S$ to $\G_n$,
$$
1\to S_{n-1}\to S \overset{p_n}\to \G_n\to 1,
$$
where $S_{n-1}$ is the kernel of the restriction of
$\phi$ to $\G_1\times\cdots\times \G_{n-1}$.  In particular,
the inductive hypothesis applies to $S_{n-1}$.

Since $\G_n$ is of type $\FP_n(\Z)$ and $H_q(S_{n-1};\Z)$ is
finitely generated for $q\le n-2$, by induction, 
on the $E^2$ page of the spectral sequence
there are only finitely generated groups in the rectangle 
 $0\le p\le n$ and $0\le q\le n-2$. It follows that all of the
groups on the $E^\infty$ page that contribute to $H_j(S;\Z)$
with $j\le n-1$ are finitely generated, with the possible exception
of that in position $(0,n-1)$.

On the $E^2$ page, the group in position $(0,n-1)$ is
$H_0(\G_n;H_{n-1}(S_{n-1};\Z))$, which is the quotient of 
$H_{n-1}(S_{n-1};\Z)$ by the action of $\G_n$. 
This action
is determined by taking a section of $p_n:S \to \G_n$
and using the conjugation action of $S$. The section
we choose is that with image $L_n\rtimes\langle t_1t_n^{-1}\rangle$.
Since $L_n$ and $t_n$ commute with $S_{n-1}$, we have
$$H_0(\G_n;H_{n-1}(S_{n-1};\Z)) = H_0(\langle t_1\rangle;H_{n-1}(S_{n-1};\Z)) \ .$$
The latter group is the $(0,n-1)$ term on the $E^2$ page
of the spectral sequence for the extension 
$$1\to S_{n-1}\to \G_1\times\cdots\times\G_{n-1}\overset{\phi}\to \Z\to 1.$$
This is a 2-column spectral sequence, so the $E^2$ page
coincides with the $E^\infty$ page. 
Since $\G_1\times\cdots\times \G_{n-1}$
is of type $\FP_{n-1}$ (indeed of type $\FP_{n}$), it follows that   
$H_0(\langle t_1\rangle;H_{n-1}(S_{n-1};\Z))$ is finitely generated, and
the induction is complete.

For the general case, replace $\Z$ by the finite index subgroup
$\phi(\G_1)\cap\cdots\cap\phi(\G_n)$ ($=m\Z$, say); replace each $\G_i$
by the finite-index subgroup $\Delta_i=\G_i\cap\phi^{-1}(m\Z)$,
and replace $S$ by the finite-index subgroup $T=S\cap(\Delta_1\times\cdots\times\Delta_n)$.  Since $\phi(\Delta_i)=m\Z$ for each $i$,
the above special-case argument applies to $T$, to show that
$H_j(T;\Z)$ is finitely generated for each $0\le j\le n-1$.  Moreover,
$T$ is normal in $S$, and we may consider the LHS spectral sequence of the short exact sequence
$$1\to T\to S\to S/T\to 1.$$
On the $E^2$ page of this spectral sequence, the terms $E^2_{pq}$
in the region $0\le q\le n-1$ are homology groups of the finite
group $T/S$ with coefficients in the finitely generated modules
$H_q(T;\Z)$, and so they are finitely generated abelian groups.
But all the terms that contribute to $H_j(S;\Z)$ for $0\le j\le n-1$
lie in this region, so $H_j(S;\Z)$ is finitely generated for $j\le n-1$,
as required.
\end{proof}

\begin{theorem}\label{thm14}
Let $\G_1,\dots,\G_n$ be non-abelian limit groups and
let $S$ be the kernel of an epimorphism $\phi:\G_1\times\cdots\times
\G_n\to\Z$. 
If the restriction of $\phi$ to each of the $\G_i$
is epic, then  $H_n(S;\Q)$ has infinite $\Q$-dimension.
\end{theorem}

\begin{proof}
The proof is by induction on $n$. 
The case $n=1$ was established in \cite{bh1}: the group $S=\ker\phi$ is a normal subgroup of the 
non-abelian limit group $\G_1$, and if $H_1(S,\Q)$ were finite dimensional
then $S$ would be finitely generated, and hence would have finite
index in $\G_1$. 

The preceding proposition shows that $H_j(S;\Z)$ is finitely generated, and hence $H_j(S;\Q)$  is
finite dimensional for $j<n$. As in the proof of that proposition, we
consider  the LHS spectral
sequence for
$$
1\to S_{n-1}\to S\overset{p_n}\to \G_n\to 1,
$$
now with $\Q$-coefficients. There are now 
only finitely generated $\Q$-modules in the region $0\le q\le n-2$. 
(Recall that $\G_i$ is of type $\FP_\infty$.)
In particular, the terms on the $E^2$ page
involved in the calculation of $H_n(S;\Q)$ are all finitely generated
except for 
$$H_0(\G_n;H_{n}(S_{n-1};\Q)) = H_0(\langle t_1\rangle;H_{n}(S_{n-1};\Q))$$
and
$$H_1(\G_n;H_{n-1}(S_{n-1};\Q)).$$

It suffices to prove that the latter is infinite dimensional
over $\Q$. (The former is actually finite dimensional, but this
is irrelevant.) 
   
The module $M=H_{n-1}(S_{n-1};\Q)$ is a homology group of 
the kernel of a map from an $\FP_\infty$ group to $\Z$.
It is thus a homology group of a chain complex of
free $R=\Q[t,t^{-1}]$ modules of finite rank. 
The ring $R$ is Noetherian, so such a homology
group is finitely generated as an $R$-module.
By the inductive hypothesis, $M$ has infinite $\Q$-dimension. 
So by the classification of finitely generated modules over a principal
ideal domain, 
$M$ has a free direct summand, that is $M=M_0\oplus R$.

The $\G_n$-action on $M$ factors through 
 the quotient $\G_n\to\G_n/L_n = \<t_n\>$, since
$L_n$ acts trivially, so the direct sum decomposition 
 passes to $M$ considered as a $\Q\G_n$ module.
Hence $H_1(\G_n;M)=H_1(\G_n;M_0) \oplus H_1(\G;R)$.

Finally, as a $\Q\G_n$ module, 
$R=\Q\G_n \otimes_{\Q L_n} \Q$, so by Shapiro's Lemma 
$H_1(\G_n;R) \cong H_1(L_n;\Q)$ 
(see for instance \cite[III.6.2. and III.5]{ksbrown}) . 
  
 As $L_n$ is an infinite index normal subgroup of a non-abelian
 limit group, it is not finitely generated, and therefore neither is
the $\Q$-module  $H_1(L_n;\Q)$ \cite{bh1}.
 \end{proof} 

Theorem \ref{theoremkernelZ} follows immediately from Theorem \ref{thm14}
in the
light of the K\"unneth formula, after one has passed to a subgroup 
of finite index to ensure that whenever $\G_i\to \Z$ is non-trivial it
is onto.

\section{Completion of the proof of the Main Theorem}\label{finalstep}

The following lemma and its corollary provide an extension to the 
virtual context of known results about finitely generated nilpotent
groups.
We shall apply them to direct products of the  
virtually nilpotent quotients of $\G_i/L_i$ resulting from 
 Theorem \ref{propvna}.

\begin{lemma}\label{nilp}
Let $G$ be a finitely generated virtually nilpotent group and 
let $\overline S$ be
a subgroup of infinite index.  Then there exists a subgroup $K$
of finite index in $G$ and an epimorphism $f:K\to\Z$ such that 
$(\overline S\cap K)\subset {\mathrm{ker}}(f)$.
\end{lemma}

\begin{pf}
We argue by induction on the Hirsch length
$h(G)$, which is strictly positive, since $G$ is infinite. 

In the initial case, $h(G)=1$ means that $G$ has an infinite cyclic
subgroup $K$ of finite index.  
Since $\overline S$ has infinite index in $G$,
$\overline S$ is finite, so $(\overline S\cap K)$ is trivial, 
and we can take $f:K\to\Z$ to be an isomorphism.

\medskip
For the inductive step, let $H$ be a finite index torsion-free
subgroup of $G$, and $C$ an infinite cyclic central subgroup of
$H$.  
If $C\-S$ has infinite index in $G$, then the inductive hypothesis
applies to $H/C$ and we are done.  
Otherwise, $\overline S$ has infinite index in $C\-S$, 
so $C\cap \overline S$ has infinite index in $C\cong\Z$.  
But then $C\cap \overline S=\{1\}$, and since $C<H$, it follows that 
$C\overline S\cap H=C\times (\overline S\cap H)$.  
Put $K=C\overline S\cap H$ and let $f$ be the projection
$K\to C$ with kernel $\overline S\cap H$. 
\end{pf}

We note that Lemma \ref{nilp}
would not remain true if one assumed only that $G$ were polycyclic. For example, it fails for
lattices $G=\Z^2\rtimes\langle t \rangle$ 
in the 3-dimensional Lie group $\rm{Sol}$
if one takes $S=\langle t\rangle$.

\smallskip
Repeated applications of Lemma \ref{nilp} yield the following.

\begin{cor}\label{chain}
Let $G$ be a finitely generated, virtually nilpotent group and
let $\overline S$ be a subgroup of $G$.  
Then there is a subnormal chain
$\-S_0<\-S_1<\cdots <\-S_r=G$, 
where $\-S_0$ is a subgroup of finite index
in $\-S$ and for each $i$ the quotient group 
$\-S_{i+1}/\-S_i$ is either finite or cyclic.
\end{cor}

 For the benefit of topologists, we should note that the following algebraic argument is
modelled on the geometric proof of the Double Coset Lemma in \cite{bh2}.

\begin{pfof}{Theorem \ref{main3}}

Let $\Gamma=\Gamma_1\times\cdots\times\Gamma_n$.
Recall that 
the $\Gamma_i$ are non-abelian, the projections $p_i:S\to\Gamma_i$
are surjective, and the intersections $L_i=S\cap\Gamma_i$
are non-trivial. 
By passing to a subgroup of finite index we may assume that each $\G_i$
splits as in Proposition \ref{assume}(5).
Let $L=L_1\times\dots\times L_n$.

We only need consider the 
case when $S$ has infinite index in $\G$.
We shall 
derive a contradiction from the assumption that for
all subgroups
$S_0<S$ of finite index  and for all $0\le j\le n$,
$H_j(S_0,\Q)$ is finite-dimensional.
 
From Theorem \ref{propvna} we know
that each of the quotient groups $\Gamma_i/L_i$ is virtually
nilpotent, and hence so is $\G/L$.

Since $L\subset S$ and $S$ has infinite index in $\G$,
 the image $\-S$ of $S$ in $\G/L$ is of infinite index
and we may apply
Lemma \ref{nilp} with $\G/L$ in the role of $G$.
Let $\Lambda<\Gamma$ be the preimage of the subgroup $K$
provided by the lemma. 
Note that $\Lambda$ has finite index in $\G$,
contains $L$, and admits an epimorphism $f:\Lambda\to\Z$
such that $S\cap \Lambda\subset\mathrm{ker}(f)$.
As in (\ref{ss:fi}), we may replace the groups
$\Gamma_i$ and $S$ by  finite-index 
subgroups so as to ensure that $L\subset S\subset N$, 
where $N$  is the
kernel of an epimorphism $\Gamma\to\Z$.
By Theorem  \ref{theoremkernelZ}, there is a finite index subgroup
$N_0<N$ and an integer $j\le n$ such that  $H_j(N_0;\Q)$ is infinite
dimensional.  

By Corollary \ref{chain} (applied to the image of $S\cap N_0$ in
$\G/L$) there is a 
subgroup $S_0$ contained in $S\cap N_0$, 
which has finite index in $S$, 
and a subnormal chain of subgroups
$S_0\triangleleft S_1\triangleleft\cdots\triangleleft S_k=N_0$
with $S_{i+1}/S_i$ either finite or cyclic for each $i$.  
We now use the following lemma to contradict  the
assumption that $H_j(S_0;\Q)$ is finite-dimensional. 

\begin{lemma}\label{chain2}
Let $S_0\triangleleft S_1$ be groups with
 $S_1/S_0$ 
finite or cyclic.
If $H_j(S_0;\Q)$ is finite dimensional for $0\le j\le n$,
then $H_j(S_1;\Q)$ is finite dimensional for  $0\le j\le n$.
\end{lemma}

\begin{proof} %%%++changed proof -- no induction
In the LHS spectral sequence for the
group extension $S_0\to S_1\to (S_1/S_0)$ we have
$E^2_{p,q}=H_p(S_1/S_0;H_q(S_0;\Q))$.
By hypothesis,
$E^2_{p,q}$ has finite $\Q$-dimension for $q\le n$.  
Moreover, $E^2_{p,q}=0$ for $p>1$, since $S_1/S_0$ 
has homological dimension at most $1$ over $\Q$.
Thus the derivatives on the $E^2$ page all vanish and the spectral sequence stabilizes at the $E^2$ page.  
Hence, for $0\le j\le n$, we have
$$\mathrm{dim}_\Q(H_j(S_1;\Q))=\mathrm{dim}_\Q(E^2_{0,j})+\mathrm{dim}_\Q(E^2_{1,j-1})<\infty,$$
as required.
\end{proof}

Repeatedly applying this lemma to the subnormal sequence 
$S_0\triangleleft S_1\triangleleft\cdots\triangleleft S_k=N_0$
implies that $H_j(N_0;\Q)$ is finite dimensional for all $j\le n$, contradicting Theorem \ref{theoremkernelZ}.

\end{pfof}

This completes the proof of Theorem \ref{main3}, 
from which Theorem \ref{main} follows immediately.

 \section{From Theorem \ref{main3} to Theorem \ref{split}} \label{s:last}

Let $\G_i,L_i$ and $S$ be as in the statement of Theorem \ref{split},
but without necessarily assuming that the $L_i$ are non-abelian for all $i$.  We first discuss how this situation differs from the special case
stated in Theorem \ref{split}.

If some $L_i$ is trivial, then $S$ is isomorphic to a subgroup
of the direct product of the $\G_j$ with $j\ne i$, as in Proposition
\ref{assume} (3).  We now assume that $L_i\ne \{1\}$ for each $i$.

As in Proposition \ref{assume} (2), we may replace each $\G_i$ by
$p_i(S)$, where $p_i:S\to\G_i$ is the projection, and
hence assume that $p_i$ is surjective, and so each $L_i$ is normal in
$\G_i$.

If some $L_i$ is non-trivial and abelian, then it is free abelian of finite
rank, by \cite[Corollary 1.23]{BF}.  Since $L_i$ is normal, it has finite
index in $\G_i$, and it follows immediately from the $\omega$-residually
free property that $\G_i$ is itself abelian.

Arguing as in Proposition \ref{assume} (4), we may assume that only one of
the $\G_i$ is abelian, say $\G_1$,  and that $L_1$ is the only
non-trivial abelian $L_i$.   We may also assume that $L_1$ is a direct 
factor of $\G_1$; say $\G_1=L_1\times M_1$.  But then $S$ virtually
splits as a direct product $L_1\times S'$, where $S'=S\cap (\G_2\times\cdots \G_n)$.

Note that the above reduction involved only one passage to a finite index subgroup,
and that was within the abelian factor $\G_1$.  The other $\G_i$ and $L_i$
are left unchanged.  In particular, the $L_i$ remain non-abelian.

We have now reduced to the situation of the statement of Theorem 
\ref{split}, with the additional hypothesis that each $p_i:S\to\G_i$
is surjective.

In particular, each  $L_i$ is normal in   $\G_i$, and hence
is of finite index   for $i=1,\dots,r$. 

Let
 $\Pi_r:\G_1\times\cdots\times\G_n\to \G_1\times\cdots\times\G_r$
 be the natural projection,
 let $\Lambda =  L_1\times\cdots\times L_r$
 and let $\hat S_0=S\cap\Pi_r^{-1}(\Lambda)$.
 Then $\hat S_0$ has finite index 
 in $S$ and $\hat S_0=\Lambda\times \hat S_2$, where
  $\hat S_2 = \hat S_0\cap(\G_{r+1}\times\cdots\times\G_n)$.
Theorem \ref{main3} now says that  that $\hat S_2$ has
 a subgroup of finite index $S_2$ with
$H_k(S_2;\mathbb Q)$  infinite dimensional for some $k\le n-r$.

\end{document}